\documentclass[a4paper,12pt]{article}

\usepackage{amsmath,amsfonts,amssymb}
\usepackage{amsthm}
\usepackage{color,multicol}

\newtheorem{lemma}{Lemma}[section]
\newtheorem{theorem}[lemma]{Theorem}

\newtheorem{remark}[lemma]{Remark}

\begin{document}

\title{An estimate for the Morse index of a Stokes wave}

\author{Eugene Shargorodsky\footnote{E-mail: \ eugene.shargorodsky@kcl.ac.uk} \\
Department of Mathematics, King's College London,\\
Strand, London WC2R 2LS, UK}

\date{}

\maketitle

\begin{abstract}
Stokes waves are steady periodic water waves on the 
free surface of an infinitely deep irrotational two dimensional flow under gravity 
without surface tension. They can be 
described in terms of solutions of the Euler-Lagrange 
equation of a certain functional. This allows one to 
define the Morse index of a Stokes wave. It is well 
known that if the Morse indices of the elements of a 
set of non-singular Stokes waves are bounded, then 
none of them is close to a singular one. The paper presents 
a quantitative variant of this result.
\end{abstract}

\section{Introduction}

A Stokes wave is a steady periodic wave, propagating under gravity with
constant speed on the surface of an infinitely deep irrotational
two dimensional flow. Its free surface is determined by the Laplace
equation, kinematic and periodic boundary conditions and by a dynamic
boundary condition given by the requirement that pressure in the flow
at the surface should be constant (Bernoulli's theorem). A mathematical 
model for Stokes waves can be described as follows. 

Let $\Omega \subset \mathbb{C}$ denote the domain below a curve $\mathcal{S}$ in the
$(X,Y)$-plane, where
\begin{eqnarray*}
&& \mathcal{S} := \{(x(s), y(s)) \,:\, \ s \in \mathbb{R}\},  \\
&& (x, y) \text{ is injective and absolutely continuous, } \\
&& x'(s)^2 + y'(s)^2 > 0  \text{ for almost all } s , \\
&& s \mapsto (x(s) - s, y(s)) \text{ is } 2\pi\text{--periodic. }
\end{eqnarray*}
The Stokes waves are solutions of the following free boundary problem:
find $\mathcal{S}$ for which there exists $\psi \in
C(\overline{\Omega})\bigcap C^2(\Omega)$ such that
\begin{eqnarray}
&& \Delta \psi = 0 \text{ in } \Omega ,
\label{Pr1}\\
&& \psi  \text{ is } 2\pi\text{--periodic in } X, \label{Pr2}\\
&& \nabla\psi(X,Y) \to (0, 1) \text{ as } Y \to -\infty \ \text{
uniformly in } X,
\label{Pr3}\\
&& \psi = 0 \text{ on } \mathcal{S}, \label{Pr4} \\
&& |\nabla\psi(X,Y)|^2  = 1 - 2\mu Y  \text{ almost everywhere on } \mathcal{S} .
\label{Pr5}
\end{eqnarray}
Here $\mu^{-1/2} > 0$ is the Froude number, a dimensionless
combination of  speed, wavelength and gravitational acceleration. If $(\mathcal{S}, \psi)$ 
is a solution of \eqref{Pr1}-\eqref{Pr5} such that $1 - 2\mu Y > 0$ everywhere on $\mathcal{S}$, then
$\mathcal{S}$ and $\psi$ are real analytic (\cite{Tol}, see also \cite{ST3}). 
We call such solutions regular or non-singular. The famous Stokes wave of extreme form which has
a stagnation point and a corner containing an angle of $120^\circ$
at its crest is a singular solution, and $1 - 2\mu Y = 0$ at the crest. Non-singular solutions
of \eqref{Pr1}-\eqref{Pr5} are in one-to-one correspondence with the critical points 
$v \in W^{1, 3}_{2\pi}$ of the
functional 
\begin{equation}\label{functional}
\mathcal{J}(v) = \mathcal{J}_\mu(v) := 
\int_{-\pi}^{\pi} \Big(v(t) \mathcal{C} v'(t) - \mu v^2(t) (1 + \mathcal{C} v'(t))\Big) dt, \ \ \ 
v \in W^{1, 2}_{2\pi}
\end{equation}
(see \cite{BDT1, ST3, Tol1}),
where $\mathcal{C} u$ denotes the periodic Hilbert transform of a
$2\pi$-periodic function $u: \mathbb{R} \to \mathbb{R}$:
\begin{equation}\label{PerHilb}
\mathcal{C} u(t) = \frac1{2\pi} \int_{-\pi}^{\pi} u(s)
\cot  \frac{t - s}2\, ds 
\end{equation}
(see Section \ref{App} for the notation of function spaces and for more information on $\mathcal{C}$).
This allows one to define the Morse index of a non-singular Stokes wave. Let $v$ be a critical
point of $\mathcal{J}$. Consider the quadratic form of the second Fr\'echet derivative 
$\mathcal{J}''(v)$ (the Hessian): 
\begin{equation}\label{Hessian}
\mathcal{Q}_v[u] := 2 \int_{-\pi}^{\pi} \Big((1 - 2\mu v(t)) u(t) \mathcal{C} u'(t) - 
\mu (1 + \mathcal{C} v'(t))u^2(t)\Big) dt .
\end{equation}
The Morse index $\mathcal{M}(v)$ of $v$ and of the corresponding Stokes wave is the number 
$N_-(\mathcal{Q}_v)$ which is defined as follows. 

Let $\mathcal{H}$ be a Hilbert space and let $\mathbf{q}$ be a Hermitian form with a domain
$\mbox{Dom}\, (\mathbf{q}) \subseteq \mathcal{H}$. Set
\begin{equation}\label{N-}
N_- (\mathbf{q}) := \sup\left\{\dim \mathcal{L}\, | \  \mathbf{q}[u] < 0, \ 
\forall u \in \mathcal{L}\setminus\{0\}\right\} ,
\end{equation}
where $\mathcal{L}$ denotes a linear subspace of $\mbox{Dom}\, (\mathbf{q})$. If $\mathbf{q}$
is the quadratic form of a self-adjoint operator $A$ with no essential spectrum in $(-\infty, 0)$, then
$N_- (\mathbf{q})$ is the number of negative eigenvalues of $A$ repeated according to their
multiplicity (see, e.g., \cite[S1.3]{BerShu} or \cite[Theorem 10.2.3]{BirSol}).

Every critical point $v \in W^{1, 3}_{2\pi}$ of \eqref{functional} is a real analytic function corresponding
to a nonsingular Stokes wave $(\mathcal{S}, \psi)$ and
$$
\min_{t \in \mathbb{R}}\, (1 - 2\mu v(t)) = \min_{(X, Y) \in \mathcal{S}} (1 - 2\mu Y) >  0 
$$
(see \cite{BDT1, ST3}). Let
\begin{eqnarray}\label{nudef}
&& \nu(v) := \max_{t \in \mathbb{R}}\, \frac{2\mu}{1 - 2\mu v(t)} =
\frac{2\mu}{\min_{t \in \mathbb{R}}\, (1 - 2\mu v(t))}\, , \nonumber \\
&& \nu_0(v) := \frac{1}{\min_{t \in \mathbb{R}}\, (1 - 2\mu v(t))} =
\frac{\nu(v)}{2\mu}\, .
\end{eqnarray}

Let $v_k \in W^{1, 3}_{2\pi}$ be a  critical
point of $\mathcal{J}_{\mu_k}$, $k \in \mathbb{N}$ (see \eqref{functional}).
It is well known that if $\nu(v_k) \to \infty$, then $\mathcal{M}(v_k) \to \infty$
as $k \to \infty$ (see \cite{BDT2, Pl, ST3}). In other words, if the Morse indices of the elements of a 
set of non-singular Stokes waves are bounded, then none of them is close to a singular one. 
The following quantitative version of this statement is the main result of the paper. 
\begin{theorem}\label{Main}
There exist constants $M_1, M_2 > 0$ such that
\begin{equation}\label{Mainest}
M_1 \ln^{1/3}(1 + \nu(v)) \le \mathcal{M}(v) \le 1 + M_2\, \nu(v) \ln(2 + \nu_0(v))
\end{equation}
holds for every critical point $v \in W^{1, 3}_{2\pi}$ of \eqref{functional} and every $\mu > 0$.
\end{theorem}

In fact, we prove in Section \ref{1.1} a more general result which holds for Bernoulli free-boundary
 problems. Plotnikov's transformation (\cite{Pl}, see also \cite{BDT2})
allows one to pass from \eqref{Hessian} to a simpler quadratic form
\begin{equation}\label{form}
\mathbf{q}_V[u] := \int_{-\pi}^{\pi} \Big((\mathcal{C} u'(t)) u(t) - V(t) u^2(t)\Big) dt , \ \ \ 
u \in H^{1/2, 2}_{2\pi} \cap C_{2\pi}\, ,
\end{equation}
where the potential $V \ge 0$, $V \in L^1_{2\pi}$ is determined by $v$. It is convenient for us
to extend the domain of $\mathbf{q}_V$ from $W^{1, 2}_{2\pi}$ to $H^{1/2, 2}_{2\pi} \cap C_{2\pi}$.
An easy approximation argument shows that this does not affect the Morse index. 

In order to state the estimate for $N_- (\mathbf{q}_V)$,
we need some notation from the theory of Orlicz  spaces (see \cite{KR, RR}).
Let $(\Omega, \Sigma, \mu)$ be a measure space, let $\Phi$ and
$\Psi$ be mutually complementary $N$-functions, and let $L_\Phi(\Omega)$, 
$L_\Psi(\Omega)$ be the corresponding Orlicz spaces. (These spaces are denoted by 
$L^*_\Phi(\Omega)$,  $L^*_\Psi(\Omega)$ in \cite{KR}, where $\Omega$ is assumed to be a closed
bounded subset of $\mathbb{R}^d$
equipped with the standard Lebesgue measure.) We will use the following
norms on $L_\Psi(\Omega)$ 
\begin{equation}\label{Orlicz}
\|f\|_{\Psi} = \|f\|_{\Psi, \Omega} = \sup\left\{\left|\int_\Omega f g d\mu\right| : \ 
\int_\Omega \Phi(g) d\mu \le 1\right\}
\end{equation}
and
\begin{equation}\label{Luxemburg}
\|f\|_{(\Psi)} = \|f\|_{(\Psi, \Omega)} = \inf\left\{\kappa > 0 : \ 
\int_\Omega \Psi\left(\frac{f}{\kappa}\right) d\mu \le 1\right\} .
\end{equation}
These two norms are equivalent
\begin{equation}\label{Luxemburgequiv}
\|f\|_{(\Psi)} \le \|f\|_{\Psi} \le 2 \|f\|_{(\Psi)}\, , \ \ \ \forall f \in L_\Psi(\Omega) .
\end{equation}
We will need the following pair of mutually complementary $N$-fuctions
\begin{equation}\label{thepair}
\mathcal{A}(s) = e^{|s|} - 1 - |s| , \ \ \ \mathcal{B}(s) = (1 + |s|) \ln(1 + |s|) - |s| , \ \ \ s \in \mathbb{R} .
\end{equation}

\begin{theorem}\label{Latour}
There exist  constants $C_1, C_2 > 0$ such that
\begin{equation}\label{est}
C_1 \|V\|_{L^1_{2\pi}} \le N_- (\mathbf{q}_V) \le C_2 
\|V\|_{\mathcal{B}, [-\pi, \pi]}  + 1,  \ \ \ 
\forall V \in L^1_{2\pi} , \ V \ge 0 . 
\end{equation}
\end{theorem}

The above theorem answers the question posed by the author at the LMS
Durham Symposium ``Operator Theory and Spectral Analysis" in 2005:
is there a function $h : \mathbb{R}_+ \to
\mathbb{R}_+$ such that $h(\tau) \to +\infty$ as $\tau \to
+\infty$ and
$$
N_- (\mathbf{q}_V) \ge
h\left(\|V\|_{L^1_{2\pi}}\right) , \ \ \  \forall V \in C^\infty_{2\pi}, \ V \ge 0 ?
$$
Although he
offered a ``decent bottle of wine" for the first correct solution, no one came up with an
answer. After repeating the question and the offer in several later talks, the 
author decided in 2007 to rise the stakes and to upgrade a ``decent bottle of wine" to 
a bottle of Chateau Latour which is arguably the best wine in the world. Although this
much more attractive offer was repeated in several talks since then and also in \cite{S2}, 
the bottle  of Chateau Latour remained unclaimed. It turns out that the latter was just pure luck
as the lower estimate in \eqref{est} can be derived from \cite{GNY} (see also
\cite{GY}).  The proof of the upper estimate in \eqref{est} is more difficult and relies upon methods developed in \cite{Sol} (see Section \ref{CLR} below).

\section{Notation and auxiliary results}\label{App} 

The subscript $2\pi$ is used throughout the paper to denote spaces of $2\pi$--periodic 
real valued functions: $C_{2\pi}$ is the Banach space of real valued continuous 
$2\pi$--periodic functions; $C^\infty_{2\pi}$ is the subspace of $C_{2\pi}$ consisting of
infinitely smooth functions;
$L^p_{2\pi}$, $p \ge 1$ is the Banach space of real valued 
locally $p^{\mathrm{th}}$--power summable $2\pi$--periodic functions;
$W^{1,p}_{2\pi}$ is the Banach space of absolutely continuous,
$2\pi$--periodic functions $u$ with weak first derivatives $u' \in L^p_{2\pi}$;
$H^{1/2, 2}_{2\pi} = W^{1/2, 2}_{2\pi}$ is the subspace of $L^2_{2\pi}$ consisting
of functions $u$ such that
$$
\|u\|_{W^{1/2, 2}_{2\pi}}^2 := \int_{-\pi}^\pi \int_{-\pi}^\pi 
\left(\frac{u(t) - u(s)}{\sin\frac{1}{2} (t - s)}\right)^2 dt\, ds +
\|u\|_{L^2_{2\pi}}^2 < +\infty .
$$
The following is an equivalent norm on $H^{1/2, 2}_{2\pi}$:
$$
\|u\|_{H^{1/2, 2}_{2\pi}} := 
\left(\sum_{n \in \mathbb{Z}} (1 + |n|) \left|\hat{u}(n)\right|^2\right)^{1/2} ,
$$
where $\hat{u}(n)$ are the Fourier coefficients of $u$ (see, e.g., \cite{Tol0}).

The periodic Hilbert transform \eqref{PerHilb} is bounded in $L^p_{2\pi}$
if $1 < p < \infty$ (M. Riesz theorem), and
\begin{eqnarray}
& \widehat{\mathcal{C}u}(0) = 0 , \ \ \ \widehat{\mathcal{C}u}(n) = -i \,\mbox{sign}(n) \hat{u}(n) , \
n \in \mathbb{Z}\setminus\{0\} , \ \ \ u \in L^p_{2\pi} , \label{CFour} \\
& \mathcal{C}^2u = - u + \hat{u}(0) . \label{C2}
\end{eqnarray}
It follows from \eqref{CFour} that
\begin{equation}\label{FirstOrd}
\widehat{\mathcal{C}u'}(n) = |n| \hat{u}(n) , \ n \in \mathbb{Z} , \ \ \ u \in W^{1,p}_{2\pi} ,
\end{equation}
and hence
$$
\left(\int_{-\pi}^{\pi} \Big((\mathcal{C} u'(t)) u(t) + u^2(t)\Big) dt\right)^{1/2}
$$
is an equivalent norm on $H^{1/2, 2}_{2\pi}$.

Let $\Psi$ be an $N$-function (\cite{KR, RR}). Then
\begin{equation}\label{LuxNormImpl}
\int_\Omega \Psi\left(\frac{f}{\kappa_0}\right) d\mu \le C_0, \ \ C_0 \ge 1  \ \ \Longrightarrow \ \
\|f\|_{(\Psi)} \le C_0 \kappa_0 . 
\end{equation}
Indeed, since $\Psi$ is even, convex and increasing on 
$[0, +\infty)$, and $\Psi(0) = 0$, we get for any $\kappa \ge C_0 \kappa_0$,
$$
\int_{\Omega} \Psi\left(\frac{f}{\kappa}\right) d\mu \le
\int_{\Omega} \Psi\left(\frac{f}{C_0 \kappa_0}\right) d\mu \le
\frac{1}{C_0} \int_{\Omega} \Psi\left(\frac{f}{\kappa_0}\right) d\mu \le 1 .
$$

We will use the following standard notation
$$
a_+ := \max\{0, a\}, \ \ \ a \in \mathbb{R} .
$$

\begin{lemma}\label{elem}
$\frac12\, s\ln_+ s \le \mathcal{B}(s) \le s + 2s\ln_+ s$, \ $\forall s \ge 0$.
\end{lemma}
\begin{proof}
Integrating the inequality 
\begin{equation}\label{123}
1 + \ln s = \ln (e s) < \ln(1 + 3s) \le 2\ln(1 + s) ,  \ \ \ s \ge 1, 
\end{equation}
one gets $s \ln_+ s \le 2\mathcal{B}(s)$.

If $s \ge 1$, then 
\begin{eqnarray*}
\mathcal{B}(s) = (1 + s) \ln(1 + s) - s \le 2s \ln(2s) - s = (2\ln 2 - 1)s + 2s\ln s \\
< s + 2s\ln s .
\end{eqnarray*}
If $s \in [0, 1)$, then integrating the inequality $\ln(1 + s) \le s$ one gets
$$
(1 + s) \ln(1 + s) - s \le \frac12\, s^2 \le \frac12\, s \le s .
$$
\end{proof}

Suppose $g \in L^\infty_{2\pi}$ and let $\kappa_0 := \|g\|_{L^1_{2\pi}}$. Then it follows from Lemma 
\ref{elem} and \eqref{123} that
\begin{eqnarray*}
&& \int_{-\pi}^\pi \mathcal{B}\left(\frac{g(t)}{\kappa_0}\right)\, dt \le  
\frac{\|g\|_{L^1_{2\pi}}}{\kappa_0} + 2  \frac{\|g\|_{L^1_{2\pi}}}{\kappa_0}
\ln_+ \frac{2\pi \|g\|_{L^\infty_{2\pi}}}{\kappa_0} \\
&& < 2\left(1 + \ln\left(1 + \frac{2\pi \|g\|_{L^\infty_{2\pi}}}{\kappa_0}\right)\right) 
< 4 \ln\left(2 + \frac{2\pi \|g\|_{L^\infty_{2\pi}}}{\|g\|_{L^1_{2\pi}}}\right) .
\end{eqnarray*}
Hence
\begin{equation}\label{1infty}
\|g\|_{\mathcal{B}, [-\pi, \pi]} \le 2\|g\|_{(\mathcal{B}, [-\pi, \pi])} \le  
8\|g\|_{L^1_{2\pi}} \ln\left(2 + \frac{2\pi \|g\|_{L^\infty_{2\pi}}}{\|g\|_{L^1_{2\pi}}}\right)
\end{equation}
(see \eqref{Luxemburgequiv} and \eqref{LuxNormImpl}).

According to Zygmund's theorem (\cite[Ch. VII, Theorem (2.8)]{Zyg}), there exist 
constants $A, B > 0$ such that
\begin{equation}\label{Zygineq}
\|\mathcal{C}f\|_{L^1_{2\pi}} \le A \int_{-\pi}^\pi |f(t)| \ln_+ |f(t)|\, dt + B
\end{equation}
for any $f$ with $f\ln_+|f| \in L^1_{2\pi}$. In fact, a necessary and sufficient condition that there exists
$B > 0$ for which \eqref{Zygineq} holds is that $A > 2/\pi$ (see \cite{Pich}). Taking 
$\kappa = \|f\|_{(\mathcal{B}, [-\pi, \pi])}$ in the inequality
$$
\left\|\mathcal{C}\left(\frac{f}{\kappa}\right)\right\|_{L^1_{2\pi}} \le 
2A \int_{-\pi}^\pi \mathcal{B}\left(\frac{f(t)}{\kappa}\right)\, dt + B
$$
(see \eqref{Zygineq} and Lemma \ref{elem}) one arrives at
\begin{equation}\label{ZygOrlineq}
\|\mathcal{C}f\|_{L^1_{2\pi}} \le (2A  + B) \|f\|_{(\mathcal{B}, [-\pi, \pi])} \le
(2A  + B) \|f\|_{\mathcal{B}, [-\pi, \pi]}\, .
\end{equation}
Sharp inequalities of this type involving other equivalent norms on $L_{\mathcal{B}}$ 
(see \cite{Ben1}) can be found in \cite{Ben2, ONW}. 

If $f \ge 0$, then
\begin{eqnarray*}
&& \int_{-\pi}^\pi (1 + f(t)) \ln(1 + f(t))\, dt \le \frac{\pi}{2}\, \|\mathcal{C}f\|_{L^1_{2\pi}} \\
&& + 2\pi \left(1 + \frac1{2\pi}\, \|f\|_{L^1_{2\pi}}\right) \ln\left(1 + \frac1{2\pi}\, \|f\|_{L^1_{2\pi}}\right) 
\end{eqnarray*}
(see \cite[Ch. VII, (2.25)]{Zyg}). Let
$$
\kappa_1 := \max\{\|\mathcal{C}f\|_{L^1_{2\pi}}, \|f\|_{L^1_{2\pi}}\}\, .
$$
Applying the above inequality to $f/\kappa_1$ one gets
$$
\int_{-\pi}^\pi \mathcal{B}\left(\frac{f(t)}{\kappa_1}\right)\, dt \le \frac{\pi}{2} +
(2\pi +1) \ln\left(1 + \frac1{2\pi}\right) =: A_0 .
$$
Hence
\begin{equation}\label{ZygOrlpos}
\|f\|_{\mathcal{B}, [-\pi, \pi]} \le
2 \|f\|_{(\mathcal{B}, [-\pi, \pi])} \le  2A_0 
\max\{\|\mathcal{C}f\|_{L^1_{2\pi}}, \|f\|_{L^1_{2\pi}}\} , \ \ \ f \ge 0
\end{equation}
(see \eqref{Luxemburgequiv} and \eqref{LuxNormImpl}).

\section{A Solomyak type estimate for two dimensional Schr\"{o}dinger operators
with singular potentials}\label{CLR}

A mapping $F$ from a metric space $(X_1, d_1)$ into
a metric space $(X_2, d_2)$ is called {\it bi-Lipschitz} if
there exists a constant $M > 0$ such that
\begin{equation*}
d_1(x, y)/M \le d_2(F(x), F(y)) \le M d_1(x, y) , \ \ \forall
x, y \in X_1 .
\end{equation*}
We say that a curve $\ell$ in $\mathbb{C}$ is a Lipshitz arc if it is
a bi-Lipschitz image of $[0, 1]$.
It is clear that a  Lipschitz arc is non-self-intersecting
and rectifiable.
Using  the arc length parametrisation, one can easily
show that a simple rectifiable curve $\ell$ is a Lipschitz arc if and only if
it is a chord-arc curve, i.e. if there exists a constant
$K \ge 1$ such that the length of the subarc of $\ell$
joining any two points is bounded by $K$ times the
distance between them. When dealing with function spaces on $\ell$, 
we will always assume that $\ell$ is equipped with the arc length measure. 

Let $\ell_j$, $j = 1, \dots, J$ be Lipschitz arcs. Any two arcs $\ell_j$, $\ell_l$ may intersect,
have common subarcs or even coincide. The most important case for us is when $J = 2$ and
$\ell_1$ and $\ell_2$ are two halves of the unit circle $\mathbb{T}$. The main result of this
section is a Cwikel-Lieb-Rozenblum type estimate for a two dimensional Schr\"{o}dinger operator
with a potential supported by $\Sigma := \cup_{j = 1}^J \ell_j$.

Let $V_j \in L_{\mathcal{B}}(\ell_j)$, $j = 1, \dots, J$ (see \eqref{thepair}), $V_j \ge 0$,  and
\begin{eqnarray}\label{SchrLip}
& \mathcal{E}_V[w] : = \int_{\mathbb{R}^2} |\nabla w(x)|^2 dx - 
\sum_{j = 1}^J \int_{\ell_j} V_j(x) w^2(x) ds(x) , & \\  
& \mbox{Dom}\, (\mathcal{E}_V) =
W^1_2\left(\mathbb{R}^2\right)\cap C\left(\mathbb{R}^2\right) . & \nonumber
\end{eqnarray}

\begin{theorem}\label{CLRth}
There exists a constant $C(\Sigma) > 0$ such that 
\begin{equation}\label{CLRest}
N_- (\mathcal{E}_V) \le C(\Sigma) \sum_{j = 1}^J \|V_j\|_{\mathcal{B}, \ell_j} + 1 , \ \ \ 
\forall V_j  \in L_{\mathcal{B}}(\ell_j) , \ V_j \ge 0 .
\end{equation}
\end{theorem}

The proof of the theorem is an adaptation of the argument in \cite{Sol} to the case of singular
potentials. In particular, we use the following
equivalent norm on $L_\Psi(\Omega)$ with $\mu(\Omega) < \infty$ which was introduced in
\cite{Sol}:
$$
\|f\|^{\rm (av)}_{\Psi} = \|f\|^{\rm (av)}_{\Psi, \Omega} = \sup\left\{\left|\int_\Omega f g d\mu\right| : \ 
\int_\Omega \Phi(g) d\mu \le \mu(\Omega)\right\} .
$$

Let $Q := (0, 1)^2$ and
$\mathbb{I} := (0, 1)$. For any subinterval $I \subseteq \mathbb{I}$, we denote the square 
$I\times(0, |I|) \subseteq Q$ by $S(I)$. We will also use the following notation:
$$
w_S := \frac{1}{|S|} \int_S w(x)\, dx ,
$$
where $S \subseteq \mathbb{R}^2$ is a set of a finite positive two dimensional 
Lebesgue measure $|S|$. 

\begin{lemma}\label{CLRl1}{\rm (Cf. \cite[Lemma 2]{Sol})}
There exists $C_3 > 0$ such that for any $I \subseteq \mathbb{I}$, any 
$w \in W^1_2(S(I))\cap C\left(\overline{S(I)}\right)$ with $w_{S(I)} = 0$ and 
any $V \in L_{\mathcal{B}}(I)$, $V \ge 0$
the following inequality holds:
\begin{equation}\label{CLRl1eq}
\int_I  V(t) w^2(t, 0)\, dt \le C_3 \|V\|^{\rm (av)}_{\mathcal{B}, I} \int_{S(I)} |\nabla w(x)|^2\, dx .
\end{equation}
\end{lemma}
\begin{proof}
Let us start with the case $I = \mathbb{I}$. There exists $C_4 > 0$ such that
$$
\left\|w^2(\cdot, 0)\right\|_{\mathcal{A}, \mathbb{I}} \le C_4 \|w\|^2_{W^1_2(Q)}, \ \ \ \forall w \in 
W^1_2(Q)\cap C\left(\overline{Q}\right)
$$
(see \eqref{thepair}). This can be proved by applying the trace theorem (see, e.g., 
\cite[Theorems 4.32 and 7.53]{Ad}) and then using the Yudovich--Pohozhaev--Trudinger
embedding theorem for $H^{1/2, 2}$ (see \cite{Gri}, \cite{Peet} and \cite[Lemma 1.2.4]{MS}, 
\cite[8.25]{Ad}) or, in one go, by applying a trace inequality of the 
Yudovich--Pohozhaev--Trudinger type (see \cite[Corollary 11.8/2]{Maz}; a sharp result can be found
in \cite{Cia}).

Next, we use the Poincar\'e inequality (see, e.g., \cite[1.11.1]{Maz}): there exists $C_5 > 0$
such that
\begin{eqnarray*}
\int_Q (w(x) - w_Q)^2\, dx =
\inf_{a \in \mathbb{R}} \int_Q (w(x) - a)^2\, dx \le C_5 \int_Q |\nabla w(x)|^2\, dx, \\ 
\forall w \in W^1_2(Q) .
\end{eqnarray*}
Hence
$$
\left\|w^2(\cdot, 0\right)\|_{\mathcal{A}, \mathbb{I}} \le C_3 \int_Q |\nabla w(x)|^2\, dx, \ \ \ \forall w \in 
W^1_2(Q)\cap C\left(\overline{Q}\right), \ w_Q = 0 ,
$$
where $C_3 = C_4(C_5 + 1)$. Now, the H\"older inequality (see \cite[Theorem 9.3]{KR})
implies
$$
\int_{\mathbb{I}}  V(t) w^2(t, 0)\, dt \le 
\|V\|_{\mathcal{B}, \mathbb{I}} \left\|w^2(\cdot, 0)\right\|_{\mathcal{A}, \mathbb{I}} \le 
C_3 \|V\|_{\mathcal{B}, \mathbb{I}}  \int_Q |\nabla w(x)|^2\, dx ,
$$
i.e. \eqref{CLRl1eq} holds for $I = \mathbb{I}$. For an arbitrary subinterval 
$I \subseteq \mathbb{I}$, \eqref{CLRl1eq} is proved by applying a simple affine transformation
and arguing as in the proof of Lemma 2 in \cite{Sol}.
\end{proof}

\begin{lemma}\label{CLRl2}{\rm (Cf. \cite[Theorem 1]{Sol})}
For any $V \in L_{\mathcal{B}}(\mathbb{I})$, $V \ge 0$ and any $n \in \mathbb{N}$ 
there exists a finite cover of $\mathbb{I}$ by
intervals $I_k$, $k = 1, \dots, n_0$  such that $n_0 \le n$ and
\begin{equation}\label{CLRl2eq}
\int_{\mathbb{I}}  V(t) w^2(t, 0)\, dt \le 4 C_3 n^{-1} \|V\|_{\mathcal{B}, \mathbb{I}} 
\int_{Q} |\nabla w(x)|^2\, dx 
\end{equation}
for all
$w \in W^1_2(Q)\cap C\left(\overline{Q}\right)$ with $w_{S(I_k)} = 0$, $k = 1, \dots, n_0$. 
(The constant $C_3$ here is the same as in Lemma \ref{CLRl1}.) 
\end{lemma}
\begin{proof}
Arguing as in the proof of Theorem 1 in \cite{Sol} (see also the proof of Lemma 7.6 in \cite{S3})
and taking into account that the best constant
in the Besicovitch covering lemma equals 2 in the one dimensional case (see, e.g., 
\cite[Theorem 7]{Jak} or \cite[Ch. 3, Section 6, Problem 2]{SS}), one can prove the existence
of a finite cover of $\mathbb{I}$ by
intervals $I_k$, $k = 1, \dots, n_0$  such that $n_0 \le n$ and
$\|V\|^{\rm (av)}_{\mathcal{B}, I_k}  = 2n^{-1} \|V\|_{\mathcal{B}, \mathbb{I}}$. Then Lemma \ref{CLRl1}
implies the following for all $w \in W^1_2(Q)\cap C\left(\overline{Q}\right)$ with 
$w_{S(I_k)} = 0$, $k = 1, \dots, n_0$
\begin{eqnarray*}
\int_{\mathbb{I}}  V(t) w^2(t, 0)\, dt \le  \sum_{k = 1}^{n_0} \int_{I_k}  V(t) w^2(t, 0)\, dt\\
\le \sum_{k = 1}^{n_0} C_3 \|V\|^{\rm (av)}_{\mathcal{B}, I_k} \int_{S(I_k)} |\nabla w(x)|^2\, dx \\
 = 2 C_3 n^{-1} \|V\|_{\mathcal{B}, \mathbb{I}} 
\sum_{k = 1}^{n_0}  \int_{S(I_k)} |\nabla w(x)|^2\, dx \\
\le 4 C_3 n^{-1} \|V\|_{\mathcal{B}, \mathbb{I}} 
\int_{Q} |\nabla w(x)|^2\, dx .
\end{eqnarray*}
In the last inequality, we use the fact that the  intervals $I_k$, $k = 1, \dots, n_0$ can be divided
into two groups in such a way that any two intervals from the same group are disjoint. 
\end{proof}

\begin{proof}[Proof of Theorem \ref{CLRth}]
Let $\varphi_j : [0, 1] \to \mathbb{R}^2$ be a bi-Lipschitz mapping such that
$\varphi_j([0, 1]) = \ell_j$, $j = 1, \dots, J$. The mapping $\varphi_j$ can be extended to a
bi-Lipschitz homeomorphism $\varphi_j : \mathbb{R}^2 \to \mathbb{R}^2$
(\cite{Tuk}, see also \cite{DP, JK, Mac} and \cite[Theorem 7.10]{Pom}). Using
Lemma \ref{CLRl2} one can easily prove that for any $n_j \in \mathbb{N}$ 
there exist
intervals $I_{j,k}$, $k = 1, \dots, n_{j,0}$  such that $n_{j, 0} \le n_j$ and
$$
\int_{\ell_j} V_j(x) w^2(x) ds(x) \le C(\ell_j) n_j^{-1} \|V_j\|_{\mathcal{B}, \ell_j} 
\int_{\mathbb{R}^2} |\nabla w(x)|^2\, dx 
$$
for all
$w \in W^1_2(\mathbb{R}^2)\cap C\left(\mathbb{R}^2\right)$ with 
\begin{equation}\label{zeroint}
\int_{\varphi_j(S(I_{j, k}))} w(x) \left|\mathbf{J}_{\varphi^{-1}_j}(x)\right| dx = 0 , \ \ \ 
k = 1, \dots, n_{j,0} ,
\end{equation}
where  $\mathbf{J}_{\varphi^{-1}_j}$ is the Jacobian determinant of 
$\varphi^{-1}_j$, and $C(\ell_j)$ is independent of $V_j$ and $n_j$.

Let $n_j = [C(\ell_j) J \|V_j\|_{\mathcal{B}, \ell_j}] + 1$, $j = 1, \dots, J$. Take any
linear subspace $\mathcal{L} \subset W^1_2(\mathbb{R}^2)\cap C\left(\mathbb{R}^2\right)$
such that 
$$
\dim \mathcal{L} > \sum_{j = 1}^J n_j = J + \sum_{j = 1}^J [C(\ell_j) J \|V_j\|_{\mathcal{B}, \ell_j}] .
$$
Since $n_{j, 0} \le n_j$, there exists $w \in \mathcal{L}\setminus\{0\}$ which satisfies 
\eqref{zeroint} for all $j = 1, \dots, J$. Then
\begin{eqnarray*}
\mathcal{E}_V[w]  &=& \int_{\mathbb{R}^2} |\nabla w(x)|^2 dx - 
\sum_{j = 1}^J \int_{\ell_j} V_j(x) w^2(x) ds(x) \\
&\ge& \int_{\mathbb{R}^2} |\nabla w(x)|^2 dx - 
\sum_{j = 1}^J \frac{C(\ell_j) \|V_j\|_{\mathcal{B}, \ell_j}}{[C(\ell_j) J \|V_j\|_{\mathcal{B}, \ell_j}] + 1} 
\int_{\mathbb{R}^2} |\nabla w(x)|^2\, dx \\
&\ge& \int_{\mathbb{R}^2} |\nabla w(x)|^2 dx - 
\sum_{j = 1}^J \frac{1}{J} 
\int_{\mathbb{R}^2} |\nabla w(x)|^2\, dx = 0 .
\end{eqnarray*}
Hence
\begin{equation}\label{withJ}
N_- (\mathcal{E}_V) \le J + \sum_{j = 1}^J [C(\ell_j) J \|V_j\|_{\mathcal{B}, \ell_j}] \le
J\left(\sum_{j = 1}^J C(\ell_j) \|V_j\|_{\mathcal{B}, \ell_j} + 1\right) .
\end{equation}

Take $R > 0$ such that 
\begin{equation}\label{inBall}
\Sigma \subset B_R := \{x \in \mathbb{R}^2 : \ |x| < R\} .
\end{equation}
Using the Yudovich--Pohozhaev--Trudinger
embedding theorem and the H\"older inequality for Orlicz spaces as in
the proof of Lemma \ref{CLRl1} one can prove the existence of a constant 
$C_0(\Sigma) > 0$ such that
\begin{eqnarray*}
\sum_{j = 1}^J \int_{\ell_j} V_j(x) w^2(x) ds(x) \le 
C_0(\Sigma) \sum_{j = 1}^J \|V_j\|_{\mathcal{B}, \ell_j}\, \|w\|^2_{W^1_2(B_R)} , \\ 
\forall w \in W^1_2(\mathbb{R}^2)\cap C\left(\mathbb{R}^2\right) .
\end{eqnarray*}
Next, we use the Poincar\'e inequality as in the proof of Lemma \ref{CLRl1}: there exists
a constant $C_R > 0$ such that
$$
\|w\|^2_{W^1_2(B_R)} \le C_R \int_{B_R} |\nabla w(x)|^2 dx ,  \ \ \
w \in W^1_2(B_R) , \ w_{B_R} = 0 . 
$$
If $\sum_{j = 1}^J \|V_j\|_{\mathcal{B}, \ell_j} \le 1/(C_0(\Sigma) C_R)$, then
$$
\mathcal{E}_V[w] \ge 0
$$
for all
$w \in W^1_2(\mathbb{R}^2)\cap C\left(\mathbb{R}^2\right)$ with $w_{B_R} = 0$, and
hence $N_- (\mathcal{E}_V) \le 1$. Combining this 
with \eqref{withJ} one gets the existence of a constant $C(\Sigma) > 0$ for which 
\eqref{CLRest} holds.
\end{proof}

\begin{remark}\label{CLRball}
{\rm Let $R > 0$ satisfy \eqref{inBall}. Then it follows from the above proof that
\eqref{CLRest} holds, with a constant depending on $R$, for the form
\begin{eqnarray*}
& \mathcal{E}_{V, R}[w] : = \int_{B_R} |\nabla w(x)|^2 dx - 
\sum_{j = 1}^J \int_{\ell_j} V_j(x) w^2(x) ds(x) , & \\  
& \mbox{Dom}\, (\mathcal{E}_{V, R}) =
W^1_2\left(B_R\right)\cap C\left(\overline{B_R}\right) . & 
\end{eqnarray*}}
\end{remark}

\begin{remark}\label{geomC}
{\rm It might be interesting to give an explicit estimate of the constant $C(\Sigma)$ in \eqref{CLRest} 
in terms of geometric properties of $\ell_j$, $j = 1, \dots, J$ 
(see \cite{CK, LLPB} for some related results in the case $V = \mbox{const}$).}
\end{remark}

\begin{remark}\label{moreCLR}
{\rm Using conformal mappings like $z \mapsto 1/(z - z_0)$ one can extend Theorem \ref{CLRth}
to some unbounded Lipschitz curves (see \cite{S3} where this trick has been applied in the
case $V \in L_{\mathcal{B}}(\mathbb{R}^2)$). One can also combine Theorem \ref{CLRth} with
the known results on Schr\"odinger operators with locally integrable on $\mathbb{R}^2$ potentials
(see \cite{BL, GN, LS, MV, MV1, S3, Sol, Sol1} and the references therein).  It would be interesting
to extend these results to potentials of the form $\varrho\mu$, where $\mu$ is a Radon measure
and $\varrho$ is a suitable function (see \cite{BEKS}).}
\end{remark}

\section{Proof of Theorem \ref{Latour}}\label{1.2}

Let
\begin{equation}\label{ourE}
\mathcal{E}_0(u, v) := \int_{-\pi}^{\pi} (\mathcal{C} u'(t)) v(t) dt , \ \ \
\mathcal{E}_0[u] :=\mathcal{E}_0(u, u) .
\end{equation}

\begin{proof}[Proof of the lower estimate in \eqref{est}]
Let $\mathbb{D}$ denote the unit disk. For any 
$w \in W^1_2(\mathbb{D})\cap C\left(\overline{\mathbb{D}}\right)$, define
$w^*(t) := w\left(e^{it}\right)$, $t \in \mathbb{R}$. According to the trace theorem (see, e.g., 
\cite[Ch. 1, Theorem 8.3]{LM} or \cite[Theorem 7.53]{Ad}), there exists $c_1 > 0$ such that
\begin{equation}\label{E0E}
\mathcal{E}_0[w^*] \le c_1 \mathcal{E}[w] , \ \ \ \forall w \in 
W^1_2(\mathbb{D})\cap C\left(\overline{\mathbb{D}}\right) ,
\end{equation}
where
$$
\mathcal{E}[w] := \|w\|^2_{W^1_2(\mathbb{D})}, \ \ \  \mbox{Dom}\, (\mathcal{E}) =
W^1_2(\mathbb{D})\cap C\left(\overline{\mathbb{D}}\right) .
$$

Let $X$ be the closed unit disk. Define a Radon measure $\sigma$ on $X$ by
$$
\sigma(E) := \int_{(E\cap\mathbb{T})_*} V(t)\, dt ,
$$ 
where
$$
F_* := \left\{t \in (-\pi, \pi] : \ e^{it} \in F\right\} , \ \ \ F \subseteq \mathbb{T} . 
$$
The corresponding Hermitian form is
$$
\sigma(w, h) := \int_X w h\, d\sigma = \int_{-\pi}^\pi V(t) w^*(t) h^*(t) dt\, .
$$
It follows from \cite[Theorem 4.1]{GNY} that there is a universal constant $c > 0$ 
for which
$$
N_-\left(\mathcal{E} - \frac{1}{c_1}\, \sigma\right) \ge \left[\frac{c}{c_1}\, \sigma(X)\right] =
\left[\frac{c}{c_1}\,  \|V\|_{L^1_{2\pi}}\right] ,
$$ 
where $[\cdot]$ denotes the integer part.

Let $\mathcal{L} \subset \mbox{Dom}\, (\mathcal{E})$ be a linear subspace such that
$$
c_1 \mathcal{E}(w) - \sigma[w] < 0 , \ \ \ \forall w \in \mathcal{L}\setminus\{0\} .
$$
Then the linear mapping $w \mapsto w^*$ is one-to-one on $\mathcal{L}$. Otherwise there would 
exist $w_0 \in \mathcal{L}\setminus\{0\}$ such that $w_0^* = 0$ and
$$
c_1 \mathcal{E}(w_0) - \sigma[w_0]  = c_1 \|w_0\|^2_{W^1_2(\mathbb{D})} - 
\int_{-\pi}^\pi V(t) \left(w_0^*(t)\right)^2 dt = c_1 \|w_0\|^2_{W^1_2(\mathbb{D})} > 0 ,
$$
which is a contradiction. Now, it follows from \eqref{E0E} that
$$
N_- (\mathbf{q}_V) \ge N_-\left(c_1 \mathcal{E} - \sigma\right) =
N_-\left(\mathcal{E} - \frac{1}{c_1}\ \sigma\right) \ge \left[\frac{c}{c_1}\,  \|V\|_{L^1_{2\pi}}\right] .
$$
Note that $N_- (\mathbf{q}_V) \ge 1$ if $V \not= 0$. Indeed, 
$$
\mathbf{q}_V[1] = \mathcal{E}_0[1] - \int_{-\pi}^\pi V(t) dt = -\|V\|_{L^1_{2\pi}} < 0 .
$$ 
Hence
$$
N_- (\mathbf{q}_V) \ge \frac{c}{2c_1}\,  \|V\|_{L^1_{2\pi}} 
$$
(cf. the proof of Theorem 4.17 in \cite{GNY}).
\end{proof}

\begin{remark}
{\rm (i) The reason for applying Theorem 4.1 of \cite{GNY} to the form $\mathcal{E}$ rather than
directly to $\mathcal{E}_0$ in the above proof is that the former is local while the latter is not.

(ii) One can take $X = \mathbb{R}^2$ instead of $X = \overline{\mathbb{D}}$ in the above proof.
The former is more convenient in the proof of the upper estimate for $N_- (\mathbf{q}_V)$
below.}
\end{remark}

\begin{proof}[Proof of the upper estimate in \eqref{est}]
Since $\left(\mathcal{E}_0[u] + \|u\|_{L^2_{2\pi}}^2\right)^{1/2}$ is an equivalent norm on 
$H^{1/2, 2}_{2\pi}$ (see \eqref{FirstOrd}, \eqref{ourE}), it follows from the trace theorem (see, e.g., 
\cite[Ch. 1, Theorems 8.1 and 8.3]{LM} or \cite[Theorems 4.28 and 7.53]{Ad}) that there exists a
linear operator $R : H^{1/2, 2}_{2\pi} \cap C_{2\pi} \to 
W^1_2\left(\mathbb{R}^2\right)\cap C\left(\mathbb{R}^2\right)$ and a constant $c_2 > 0$
which satisfy
$$
\|Ru\|^2_{W^1_2\left(\mathbb{R}^2\right)} \le c_2 \left(\mathcal{E}_0[u] + \|u\|^2_{L^2_{2\pi}}\right) , 
\ \ \  \forall u \in H^{1/2, 2}_{2\pi} \cap C_{2\pi} ,
$$
and $(Ru)\left(e^{it}\right) = u(t)$, $t \in \mathbb{R}$.

Consider the form
\begin{eqnarray*}
& \mathcal{E}_{1,V}[w] : = \int_{\mathbb{R}^2} |\nabla w(x)|^2 dx - 
c_2 \int_{-\pi}^{\pi} (V(t) + 1) w^2\left(e^{it}\right) dt , & \\  
& \mbox{Dom}\, (\mathcal{E}_{1,V}) =
W^1_2\left(\mathbb{R}^2\right)\cap C\left(\mathbb{R}^2\right) . &
\end{eqnarray*}

Let $\mathcal{L}_0$ be a linear subspace of $H^{1/2, 2}_{2\pi} \cap C_{2\pi}$ such that
$\dim \mathcal{L}_0 = N_- (\mathbf{q}_V)$ and
$$
\mathbf{q}_V[u] < 0, \ \ \  \forall u \in \mathcal{L}_0\setminus\{0\}
$$
(see \eqref{N-}, \eqref{form}), and let $\mathcal{L} = R\mathcal{L}_0$. Then
$$
\mathcal{E}_{1,V}[w]  < 0, \ \ \  \forall w \in \mathcal{L}\setminus\{0\} .
$$
Indeed,
\begin{eqnarray*}
\mathcal{E}_{1,V}[w] : = \int_{\mathbb{R}^2} |\nabla w(x)|^2 dx - 
c_2 \int_{-\pi}^{\pi} (V(t) + 1) w^2\left(e^{it}\right) dt \\
\le c_2 \left(\mathcal{E}_0[w^*] + \|w^*\|^2_{L^2_{2\pi}}\right) - 
c_2 \int_{-\pi}^{\pi} (V(t) + 1) w^2\left(e^{it}\right) dt \\
= c_2 \mathbf{q}_V[w^*] < 0, \ \ \  \forall w \in \mathcal{L}\setminus\{0\} ,
\end{eqnarray*}
where $w^*(t) := w\left(e^{it}\right)$, $t \in \mathbb{R}$ as in the proof of the lower estimate 
in \eqref{est}. Since $\dim \mathcal{L} = \dim \mathcal{L}_0 = N_- (\mathbf{q}_V)$, Theorem
\ref{CLRth} implies the existence of a constant $C > 0$ such that
\begin{eqnarray}\label{largeV}
N_- (\mathbf{q}_V) \le N_- (\mathcal{E}_{1,V}) \le C \|c_2(V + 1)\|_{\mathcal{B}, [-\pi, \pi]}  + 1 
\nonumber \\
\le C c_2 \|V\|_{\mathcal{B}, [-\pi, \pi]} + C c_2 \|1\|_{\mathcal{B}, [-\pi, \pi]} + 1 . 
\end{eqnarray} 
(According to \cite[(9.11)]{KR}, $\|1\|_{\mathcal{B}, [-\pi, \pi]} = 
2\pi \mathcal{A}^{-1}\left(\frac{1}{2\pi}\right)$.)

Using the Yudovich--Pohozhaev--Trudinger
embedding theorem for $H^{1/2, 2}$ and the H\"older inequality for Orlicz spaces as in
the proofs of Lemma \ref{CLRl1} and Theorem \ref{CLRth}, one can prove the existence of 
a constant $c_3 > 0$ such that
$$
 \int_{-\pi}^{\pi} V(t)  u^2(t) dt \le c_3 \|V\|_{\mathcal{B}, [-\pi, \pi]} \left(\mathcal{E}_0[u] + 
 \|u\|^2_{L^2_{2\pi}}\right) , \ \ \  \forall u \in H^{1/2, 2}_{2\pi} \cap C_{2\pi} .
$$
If $\|V\|_{\mathcal{B}} \le 1/(2c_3)$, then
$$
\mathbf{q}_V[u] \ge \mathcal{E}_2[u] := \frac12 \left(\mathcal{E}_0[u] -  \|u\|^2_{L^2_{2\pi}}\right) , 
\ \ \  \forall u \in H^{1/2, 2}_{2\pi} \cap C_{2\pi} 
$$
and
$$
N_- (\mathbf{q}_V) \le N_- (\mathcal{E}_2) = 1 ,
$$
where the last equality holds because the spectrum of $u \mapsto \mathcal{C}u'$ consists of 
a simple eigenvalue $0$ and of double eigenvalues $m \in \mathbb{N}$ (see \eqref{FirstOrd}). 
Combining this estimate
with \eqref{largeV} one gets the existence of a constant $C_2 > 0$ for which the upper estimate 
in \eqref{est} holds.
\end{proof}

\begin{proof}[Sketch of an alternative proof of the upper estimate in \eqref{est}]
Let $\mathcal{P} u$ be the harmonic in $\mathbb{D}$ function such that
$(\mathcal{P} u)\left(e^{it}\right) = u(t)$, $t \in \mathbb{R}$, i.e. let $\mathcal{P} u$
be the Poisson integral of $u$. If $u \in C^\infty_{2\pi}$, then the Cauchy--Riemann equations
allow one to express the normal derivative of $\mathcal{P} u$ at $z = e^{it}$ in terms of the
tangential derivative of the harmonic conjugate of $\mathcal{P} u$. Applying Green's identity
one gets
$$
\int_{\mathbb{D}} |\nabla \mathcal{P} u(x)|^2 dx = \mathcal{E}_0[u] ,
$$ 
which can be extended to all $u \in H^{1/2, 2}_{2\pi}$ by continuity. 

Let us extend $\mathcal{P} u$ to $\mathbb{R}^2 = \mathbb{C}$ by
$$
\mathcal{P} u(z) := \mathcal{P} u\left(\frac{1}{\bar{z}}\right) , \ \ \ z \in \mathbb{C} , \ |z| > 1 .
$$
An easy calculation shows that 
$$
\int_{\mathbb{R}^2\setminus\mathbb{D}} |\nabla \mathcal{P} u(x)|^2 dx = 
\int_{\mathbb{D}} |\nabla \mathcal{P} u(x)|^2 dx .
$$ 
Since 
$$
\lim_{\mathbb{R}^2\setminus\overline{\mathbb{D}}\ni y \to x} \mathcal{P} u(y) =
\lim_{\mathbb{D}\ni y \to x} \mathcal{P} u(y) , \ \ \ |x| = 1,
$$
the extended function $\mathcal{P} u$ belongs to 
$W^1_2\left(B_2\right)\cap C\left(\mathbb{R}^2\right)$
(see \eqref{inBall}) and
$$
\int_{B_2} |\nabla \mathcal{P} u(x)|^2 dx \le
\int_{\mathbb{R}^2} |\nabla \mathcal{P} u(x)|^2 dx = 2 \mathcal{E}_0[u] .
$$ 
Let
\begin{eqnarray*}
& \mathcal{E}_{3,V}[w] : = \int_{B_2} |\nabla w(x)|^2 dx - 
2 \int_{-\pi}^{\pi} V(t) w^2\left(e^{it}\right) dt , & \\  
& \mbox{Dom}\, (\mathcal{E}_V) =
W^1_2\left(B_2\right)\cap C\left(\overline{B_2}\right) . &
\end{eqnarray*}
Using Remark \ref{CLRball} and arguing as above, one gets
$$
N_- (\mathbf{q}_V) \le N_- (\mathcal{E}_{3,V}) \le 2C \|V\|_{\mathcal{B}, [-\pi, \pi]}  + 1 ,
$$
and there is no need for an additional treatment of the case of a small $\|V\|_{\mathcal{B}, [-\pi, \pi]}$
like in the final part of the above proof. 
\end{proof}

\begin{remark}\label{noL1upper}
{\rm Repeating the argument from Example 2.7 in \cite{GN} one can show that no estimate of the type 
$$
N_- (\mathbf{q}_V) \le  \mbox{const} + \int_{-\pi}^{\pi} V(t)  W(t)\, dt
$$
can hold, provided the weight function $W$ is bounded in a neighborhood of at least 
one point.}
\end{remark}

\begin{remark}\label{asympt}
{\rm If $V \in L_{\mathcal{B}}([-\pi, \pi])$, $V \ge 0$, then
\begin{equation}\label{asymform}
\lim_{\alpha \to +\infty} \frac{N_- (\mathbf{q}_{\alpha V})}{\alpha} = \frac{1}{\pi}
\int_{-\pi}^{\pi} V(t) \, dt .
\end{equation}
Indeed, suppose first that $V \in C^\infty_{2\pi}$, $V > 0$ and consider the form
$$
\mathcal{E}_{\alpha,V}[u] :=  \mathcal{E}_0\left[V^{-1/2} u\right] 
- \alpha  \|u\|^2_{L^2_{2\pi}}, \ \ \ 
u \in H^{1/2, 2}_{2\pi} \cap C_{2\pi}\, .
$$
Since the operator of multiplication by $V^{-1/2}$ is invertible on $H^{1/2, 2}_{2\pi} \cap C_{2\pi}$,
we get
$$
N_- (\mathbf{q}_{\alpha V}) = N_-(\mathcal{E}_{\alpha,V}) .
$$
The right-hand side equals the value at $\alpha$ of the eigenvalue counting function of
the operator $u \mapsto V^{-1/2} \mathcal{C}\left(V^{-1/2} u\right)'$, which can be viewed as 
a first order pseudodifferential operator on the unit circle. Formula \eqref{asymform} follows
from the well known results on spectral asymptotics for elliptic pseudodifferential operators
on compact manifolds (see, e.g., \cite[\S 15]{Shu}; sharp asymptotic results for operators on the
circle can be found in \cite{Roz}). One can extend \eqref{asymform} to arbitrary
$V \in L_{\mathcal{B}}([-\pi, \pi])$, $V \ge 0$ by a standard perturbation-theoretic argument
(see, e.g., the proof of Theorem 2.2 in \cite{Sol1}). }
\end{remark}

\section{Proof of Theorem \ref{Main}}\label{1.1}

We prove a generalisation of Theorem \ref{Main} to Bernoulli free-boundary problems
(\cite{ST3})
which are obtained from \eqref{Pr1}-\eqref{Pr5} by substituting the boundary condition \eqref{Pr5} 
with the following one:
\begin{equation}\label{Pr5'}
|\nabla\psi(X,Y)|^2  = \lambda(Y)  \text{ almost everywhere on } \mathcal{S} ,
\end{equation}
where $\lambda :\mathbb{R} \to \mathbb{R}$ is a given continuous function. We assume that
$\lambda$ is non-constant and real analytic on the open set of full measure where it is non-zero,
and that the following holds with some $\varrho > 0$:
\begin{eqnarray}
&& \mbox{if } \lambda(y_0) = 0, \ \mbox{ then } \ \ |\lambda(y)| \le \mbox{const}\, |y - y_0|^\varrho , \ \ \
\forall y \in \mathbb{R} , \label{varrho} \\
&& \ln |\lambda|  \mbox{ is concave, and }  \lambda' \le 0 \mbox{ where } \lambda \not= 0 .
\label{logconc} 
\end{eqnarray}
If $(\mathcal{S}, \psi)$ is a solution of \eqref{Pr1}-\eqref{Pr4}, \eqref{Pr5'}, then, clearly,   
$\lambda \ge 0$ almost everywhere on $\mathcal{S}$. If
$\lambda > 0$ everywhere on $\mathcal{S}$, then
$\mathcal{S}$ and $\psi$ are real analytic (\cite{ST3}), and the
solution is called regular or non-singular.  Let
$$
p(\varrho) := \frac{\varrho + 2}{\varrho}\, .
$$
Non-singular solutions
of \eqref{Pr1}-\eqref{Pr4}, \eqref{Pr5'} are in one-to-one correspondence with the critical points 
$v \in W^{1, p(\varrho)}_{2\pi}$ of the functional 
\begin{equation}\label{Bernfunctional}
\mathcal{J} (v) = \int_{-\pi}^{\pi} \Big(\Lambda(v(t)) (1 + \mathcal{C} v'(t)) - v(t)\Big) dt, \ \ \ 
v \in W^{1, 2}_{2\pi} ,
\end{equation}
where $\Lambda$ is a primitive of $\lambda$ (see \cite{ST3}). 

Let $v$ be a critical point of $\mathcal{J}$.
The Morse index $\mathcal{M}(v)$ of $v$ and of the corresponding solution to 
\eqref{Pr1}-\eqref{Pr4}, \eqref{Pr5'} is the number 
$N_-(\mathcal{Q}_v)$, where $\mathcal{Q}_v$ is the quadratic form of the second 
Fr\'echet derivative $\mathcal{J}''(v)$ (the Hessian): 
\begin{equation}\label{BernHessian}
\mathcal{Q}_v[u] :=  \int_{-\pi}^{\pi} \Big(2\lambda(v(t)) u(t) \mathcal{C} u'(t) +
\lambda'(v(t)) (1 + \mathcal{C} v'(t))u^2(t)\Big) dt .
\end{equation}

Every critical point $v \in W^{1, p(\varrho)}_{2\pi}$ of \eqref{Bernfunctional} is 
a real analytic function and
$$
\min_{t \in \mathbb{R}}\, \lambda(v(t))  >  0 
$$
(see \cite{ST3}). Let
\begin{equation}\label{Bernnudef}
\nu(v) := \max_{t \in \mathbb{R}}\, \frac{|\lambda'(v(t))|}{\lambda(v(t))}\, , \ \ \ 
\nu_0(v) := \max_{t \in \mathbb{R}}\, \frac{1}{\lambda(v(t))} =
\frac{1}{\min_{t \in \mathbb{R}}\, \lambda(v(t))}\, .
\end{equation}

Suppose there exist constants $m_1, m_2 > 0$ such that
\begin{equation}\label{strvarrho}
\frac{m_1}{\lambda(y)^{1/\varrho}} \le \frac{|\lambda'(y)|}{\lambda(y)} \le
\frac{m_2}{\lambda(y)^{1/\varrho}} \ \ \mbox{ for all } \ y \in \mathbb{R} \ 
\mbox{ with } \ \lambda(y) \not= 0.
\end{equation}
It is easy to see that \eqref{strvarrho} implies \eqref{varrho}. 

In the case of Stokes waves, $\lambda(y) = 1 - 2\mu y$ and \eqref{logconc}, \eqref{strvarrho}
hold with $\varrho = 1$ and $m_1 = m_2 = 2\mu$ (note also that $p(1) = 3$). Hence, Theorem
\ref{Main} is a special case of the following result.
\begin{theorem}\label{BernMain}
Suppose \eqref{logconc} and \eqref{strvarrho} hold.
Then there exist constants $M_1, M_2 > 0$ which depend only on $\varrho$ and are such that
\begin{equation}\label{BernMainest}
M_1 \frac{m_1}{m_2} \ln^{\frac{\varrho}{\varrho + 2}}(1 + \nu(v)) \le 
\mathcal{M}(v) \le 1 + M_2\, \nu(v) \ln(2 + \nu_0(v))
\end{equation}
holds for every critical point $v \in W^{1, p(\varrho)}_{2\pi}$ of \eqref{Bernfunctional}. 
\end{theorem}

Using Plotnikov's transformation, one can show that $\mathcal{M}(v) =
N_- (\mathbf{q}_V)$ (see \eqref{form}), where
\begin{equation}\label{PlotPot}
2 V(t) = \mathcal{C}\left(\frac{\lambda'(v(t))}{\lambda(v(t))}\, v'(t)\right) -
\frac{\lambda'(v(t))}{\lambda(v(t))} (1 + \mathcal{C} v'(t)) > 0 , \ \ \ t \in \mathbb{R} 
\end{equation}
(see \cite{ST3}). Hence, Theorems \ref{BernMain} and \ref{Main} follow from Theorem
\ref{Latour} if one has suitable estimates for $\|V\|_{L^1_{2\pi}}$ and 
$\|V\|_{\mathcal{B}, [-\pi, \pi]}$. These are provided by the following two lemmas.

\begin{lemma}\label{ST3est}
There exists a constant $M_3 > 0$ which depends only on $\varrho$ and is such that
$$
\|V\|_{L^1_{2\pi}} \ge M_3 \frac{m_1}{m_2} \ln^{\frac{\varrho}{\varrho + 2}}(1 + \nu(v)) . 
$$
\end{lemma}
\begin{proof}
This is an easy corollary of Lemma 4.25 in \cite{ST3}.
\end{proof}

\begin{lemma}\label{ST3LlogL}
There exists a constants $M_4 > 0$ which depends only on $\varrho$ and is such that
$$
\|V\|_{\mathcal{B}, [-\pi, \pi]} \le M_4\, \nu(v) \ln(2 + \nu_0(v)) . 
$$
\end{lemma}
\begin{proof}
Since  $v \in W^{1, p(\varrho)}_{2\pi}$ is a critical point of 
\eqref{Bernfunctional}, $1 + \mathcal{C} v' > 0$ and 
$\|1 + \mathcal{C} v'\|_{L^\infty_{2\pi}} \le \sqrt{\nu_0(v)}$ 
(see \cite[Theorems 2.4, 3.3 and 3.5]{ST3}).
Then it follows from \eqref{PlotPot}, \eqref{CFour} and \eqref{Bernnudef} that
$$
\|V\|_{L^1_{2\pi}} = \int_{-\pi}^\pi V(t)\, dt = 
 \int_{-\pi}^\pi \frac{|\lambda'(v(t))|}{2\lambda(v(t))} (1 + \mathcal{C} v'(t))\, dt \le
\pi \nu(v)  
$$
(see \cite[Lemma 4.24]{ST3}), and \eqref{1infty} implies
\begin{eqnarray*}
\|1 + \mathcal{C} v'\|_{\mathcal{B}, [-\pi, \pi]} \le  
8\ln\left(2 + \frac{2\pi\, \sqrt{\nu_0(v)}}{\|1 + \mathcal{C} v'\|_{L^1_{2\pi}}}\right)  
\int_{-\pi}^\pi (1 + \mathcal{C} v'(t))\, dt \\
= 16\pi \ln\left(2 + \sqrt{\nu_0(v)}\right) . 
\end{eqnarray*}
Hence
\begin{eqnarray*}
\|\mathcal{C} V\|_{L^1_{2\pi}} \le  
 \frac12 \left\|\frac{\lambda'(v)}{\lambda(v)}\, v'\right\|_{L^1_{2\pi}} + \frac12
\left\|\mathcal{C}\left(\frac{\lambda'(v)}{\lambda(v)} (1 + \mathcal{C} v')\right)\right\|_{L^1_{2\pi}} \\
\le \frac{\nu(v)}2\, \|v'\|_{L^1_{2\pi}} + \mbox{const}\, \nu(v) \|1 + \mathcal{C} v'\|_{\mathcal{B}, [-\pi, \pi]} \\
\le \mbox{const}\, \nu(v) \|1 + \mathcal{C} v'\|_{\mathcal{B}, [-\pi, \pi]} \le
\mbox{const}\, \nu(v) \ln(2 + \nu_0(v))
\end{eqnarray*}
with constants independent of $v$ and $\lambda$ (see \eqref{C2} and \eqref{ZygOrlineq}). 
It is now left to apply \eqref{ZygOrlpos}.
\end{proof}

\end{document}